\documentclass[12pt]{article}
\usepackage{amsthm,amsmath,amssymb,amsfonts,dsfont,cases}
\advance \topmargin by -0.8in
\advance \oddsidemargin by -0.20in
\advance \evensidemargin by -0.20in

\usepackage{amsmath}
\usepackage{amssymb}
\usepackage{indentfirst}
\usepackage{verbatim}
\usepackage{multirow}
\usepackage{array}
\usepackage{graphicx}
\usepackage{subfigure}
\usepackage{indentfirst}
\usepackage{fancyhdr}
\usepackage{booktabs, longtable}
\usepackage{cite}
\usepackage{enumerate}
\usepackage{float}
\usepackage{hyperref}
\usepackage{caption}
\usepackage{hypcap}
\usepackage{xcolor}
\usepackage{color}
\DeclareCaptionLabelSeparator{twospace}{.}
\usepackage{amsthm}

\newtheorem{theorem}{Theorem}[section]
\newtheorem{lemma}{Lemma}[section]

\newtheorem{remark}{Remark}[section]

\theoremstyle{plain}
\numberwithin{equation}{section}
\begin{document}
\date{}
%\date{\today}

\title{Wave propagation for a discrete diffusive vaccination epidemic model with bilinear incidence
}
\author{
{Ran Zhang}\\
Department of Mathematics\\ Nanjing University of Aeronautics and Astronautics\\ Nanjing 211106, China\\
{Shengqiang Liu\thanks{Corresponding Author: sqliu@tiangong.edu.cn}}\\
School of Mathematical Sciences, Tiangong University\\ Tianjin 300387, China
}
\maketitle

\baselineskip=14pt

\begin{abstract}
The aim of the current paper is to study the existence of traveling wave solutions (TWS) for a vaccination epidemic model with bilinear incidence. The existence result is determined by the basic reproduction number $\Re_0$. More specifically, the system admits a nontrivial TWS when $\Re_0>1$ and $c \geq \mathfrak{c}^*$, where $\mathfrak{c}^*$ is the critical wave speed. We also found that the TWS is connecting two different equilibria by constructing Lyapunov functional. Lastly, we give some biological explanations from the perspective of epidemiology.
\end{abstract}

%%%%%%%%%%%%%%%%%%%%%%%%%%% section 1

\section{Introduction}\label{sec:Inc}
\def\d {{\rm d}}
Vaccination is critical for the prevention and control of infectious diseases, there are more than 20 life-threatening diseases could be prevented by vaccines up to now.
Vaccinators can achieve immunity by having the immune system recognize foreign substances, antibodies are then screened and generated to produce antibodies against the pathogen or similar pathogen, and then giving the injected individual a high level of disease resistance.
In \cite{LiuTakeuchiShingoJTB2008}, Liu et al. proposed the following system with continuous vaccination strategy:
\begin{equation}
\label{PrModel1}\left\{
\begin{array}{l}
\vspace{2mm}
\displaystyle   \frac{{\rm d} S(t)}{{\rm d} t} = \Lambda - \beta_1 S(t)I(t)  - \alpha S(t) - \mu S(t),\\
\vspace{2mm}
\displaystyle   \frac{{\rm d} V(t)}{{\rm d} t} = \alpha S(t) - \beta_2 V(t)I(t)  - (\gamma_1 + \mu) V(x,t),\\
\vspace{2mm}
\displaystyle   \frac{{\rm d} I(t)}{{\rm d} t} = \beta_1 S(t)I(t) + \beta_2 V(t)I(t) - \gamma I(x,t) - \mu I(x,t),\\
\displaystyle   \frac{{\rm d} R(t)}{{\rm d} t} = \gamma_1V(t) + \gamma I(t)- \mu R(t),
\end{array}\right.
\end{equation}
where $S(t),\ V(t), \ I(t)$ and $R(t)$ are the densities of susceptible, vaccinated, infective and removed individuals at time $t$, respectively.
The parameters of model (\ref{PrModel1}) are biologically explained as in Table \ref{tab}.
\begin{table}[h]
\centering
\begin{tabular}{cl}
\hline
Parameter & \hspace{0.5cm}Interpretation \\
\hline\hline
$\Lambda$ &  Recruitment rate \\
$\beta_1$ &  Disease transmission rate between infectious and susceptible individuals \\
$\beta_2$ &  Disease transmission rate between infectious and vaccinated individuals \\
$\alpha$ &  The vaccinated rate \\
$\mu$ &  Natural death rate \\
$\gamma$ &  Recovery rate \\
$\gamma_1$ &  Rate at which a vaccinating individual obtains immunity \\
\hline \\
\end{tabular}
\caption{Biological meaning of parameters in model (\ref{PrModel1}).} %All these constants are assumed to be positive.}
\label{tab}
\end{table}
In \cite{LiuTakeuchiShingoJTB2008}, the authors shown that the disease-free equilibrium for model (\ref{PrModel1}) is globally asymptotically stable (GAS) if the basic reproduction number is less than one, while if the number is greater than one, then a positive endemic equilibrium exists which is GAS.
Since then, the epidemic models with vaccination have attracted the attention of many scholars.
Kuniya \cite{KuniyaNARWA2013} extended the study in \cite{LiuTakeuchiShingoJTB2008} to a multi-group case, and then studied the global stability by using the graph-theoretic approach and Lyapunov method.
Considering the effect of age, three vaccination epidemic models with age structure are proposed in \cite{DuanYuanLiAMC2014,WangZhangKuniyaIMAJAM2016,WangGuoLiuIMAJAM2017}, and the global stabilities are studied.
For more recent studies on the vaccination epidemic models, we refer to \cite{WangWangZhangMMAS2022,HuoJAAC} and the references therein.

With the increasing trend of globalization and mobility of people, the spatial structure of human density and location has a significant impact on the spread of diseases. It is necessary to investigate the role of diffusion in the epidemic modeling. Mathematically, Laplacian operator in the reaction-diffusion systems could can be well used to study the infectious disease model with diffusion, since it could describe the random diffusion of each individual in the adjacent space. On the other hand, nonlocal operator could describe the long range diffusion on the whole habitat \cite{LiLiYangAMC2014}.
In the study of local and nonlocal diffusive epidemic models, there is a solution called traveling wave solution (TWS). Viewing from infectious diseases perspective, the existence of TWS for epidemic model implies that the disease can be invaded \cite{LiLiLinCPAA2015}.
Up to now, there have been many studies on the TWS for local and nonlocal diffusive epidemic models (see, for example, \cite{Hosono,DucrotMagalNon2011,WangXuJAAC2014,WangMa2017,XuLiLinJDE2019,ZhangLiuFengJinAML2021}).
By considering both vaccination and spatial diffusion, Xu et al. \cite{XuXuHuangCMA2018} studied a local diffusive SVIR model, where the global dynamics on bounded domain and TWS on unbounded domain for the model were studied.
Meanwhile, the problem of TWS for two different SVIR models with nonlocal diffusion were investigated in \cite{ZhangLiuMBE2019,ZhouYangHsuDCDSB2020}.

Unlike local and nonlocal diffusive, there is another diffusion in infectious disease modeling, which is discrete diffusion. In fact, epidemic model with discrete diffusion can be regarded as lattice system, such system is better to describe the epidemic model with patch structure \cite{SanHeCPAA2021}. Recently, Chen et al. \cite{ChenGuoHamelNon2017} proposed a lattice SIR epidemic model:
\begin{equation}
\label{PreModel}\left\{
\begin{array}{l}
\vspace{2mm}
\displaystyle   \frac{\d S_n}{\d t}= [S_{n+1} + S_{n-1} - 2S_n] + \mu -  \beta S_nI_n - \mu S_n,\\
\displaystyle   \frac{\d I_n}{\d t}= d[I_{n+1} + I_{n-1} - 2I_n] + \beta S_nI_n  - (\gamma+\mu) I_n,
\end{array}\right.
\end{equation}
where $n\in\mathbb{Z}$. $S_n$ and $I_n$ denote densities of susceptible and infectious individuals at time $t$ and niche $n$. $\beta$ is the disease transmission rate. $1$ (normalized) and $d$ denote the random migration parameters for each compartments.
Chen et al. shown that system (\ref{PreModel}) admits TWS when $\Re_0>1$ and $c\geq c^*$. More recently, the TWS for (\ref{PreModel}) was proved to be converged to the endemic equilibrium by Zhang et al \cite{ZhangLiuDCDSB2021,ZhangWangLiuJNS2021}.
Model (\ref{PreModel}) is an SIR model with constant recruitment (i.e. the constant $\Lambda$), and
the existence of TWS for the discrete diffusive epidemic model without constant recruitment was studied in \cite{FuGuoWuJNCA2016,ZhangWuIJB2019,ZhangYuMMAS2021}. However, to our best knowledge, there are only a few studies focus on the problem of TWS for discrete diffusive epidemic models, especially for the model with constant recruitment.

Based on the above facts, in order to study the role of vaccination and patch structure in the disease modeling, we consider a discrete diffusive vaccination model as follows
\begin{equation}
\label{Model}\left\{
\begin{array}{l}
\vspace{2mm}
\displaystyle   \frac{\d S_n}{\d t}= [S_{n+1} - 2S_n + S_{n-1} ] + \Lambda -  \beta_1S_nI_n  - \alpha S_n - \mu S_n,\\
\vspace{2mm}
\displaystyle   \frac{\d V_n}{\d t}= [V_{n+1} - 2V_n + V_{n-1} ] + \alpha S_n -  \beta_2V_nI_n - \gamma_1 V_n - \mu V_n,\\
\vspace{2mm}
\displaystyle   \frac{\d I_n}{\d t}= d[I_{n+1} - 2I_n + I_{n-1} ] + \beta_1S_nI_n + \beta_2V_nI_n - \gamma I_n - \mu I_n,\\
\displaystyle   \frac{\d R_n}{\d t}= [R_{n+1} - 2R_n + R_{n-1} ] + \gamma_1 V_n + \gamma I_n - \mu R_n,
\end{array}\right.
\end{equation}
%where $S_n$, $V_n$, $I_n$ and $R_n$ denote the densities of susceptible, vaccinated, infectious and removed individuals at niche $n$ at time $t$.
where $S_n$, $V_n$, $I_n$ and $R_n$ denote susceptible, vaccinated, infectious and removed individuals.
$d$ is the spatial motility of infectious individuals and the diffusive rate of other compartments are normalized to be 1. The biological significance of the parameters of (\ref{Model}) are the same as those in (\ref{PrModel1}).

The current paper devotes to study the existence of TWS for system (\ref{Model}) with bilinear incidence.
In fact, there are very few studies on TWS for the epidemic model with bilinear incidence and the main difficulty is the boundedness of TWS \cite{SanHeCPAA2021}.
On the other hand, introducing the constant recruitment (i.e. $\Lambda$ in model (\ref{Model})) will bring much more complexity in mathematical analysis than the system without constant recruitment.
Moreover, it is difficult to obtain the behaviour of TWS at $+\infty$ for such model (see, for example, \cite{ChenGuoHamelNon2017}).
One motivation of this paper is to show the convergence of TWS for lattice epidemic model (\ref{Model}).
To gain this purpose,
we will construct an appropriate Lyapunov functional for the wave form equations corresponding to lattice dynamical system (\ref{Model}).
To do this, we prove the persistence of TWS, which is crucial to guarantee the Lyapunov functional has a lower bound.
We should be point out that, for different models, the construction of Lyapunov functional is also different and requires technique.
Biologically, since the vaccination has an effect of decreasing the basic reproduction number in \cite{LiuTakeuchiShingoJTB2008}, we want to study how vaccination affects the speed of TWS.

The organization of this paper is as follows. In section \ref{Sec:Pre}, we give some preliminaries results. Section \ref{Sec:Existence} devote to study the existence of TWS of system (\ref{Model}) by applying Schauder's fixed point theorem. In Section \ref{Sec:Bound}, we show the boundedness of TWS. Furthermore, we show the convergence of TWS in Section \ref{Sec:Lyapunov}. Finally, there is a brief discussion and some explanations from the perspective of epidemiology will be given in Section \ref{Sec:Dis}.

\section{Preliminaries}\label{Sec:Pre}
Firstly, the corresponding ODE system for (\ref{Model}) is
\begin{equation}
\label{ODEModel}\left\{
\begin{array}{l}
\vspace{2mm}
\displaystyle   \frac{\d S}{\d t}= \Lambda -  \beta_1SI - \mu_1 S,\\
\vspace{2mm}
\displaystyle   \frac{\d V}{\d t}= \alpha S -  \beta_2VI  - \mu_2 V,\\
\displaystyle   \frac{\d I}{\d t}= \beta_1SI + \beta_2VI  - \mu_3 I,
\end{array}\right.
\end{equation}
where $R$ equation is decoupled from other equations.
Clearly, system (\ref{ODEModel}) has a disease-free equilibrium $E_0 = (S_0,V_0,0) = \left(\frac{\Lambda}{\mu_1}, \frac{\Lambda\alpha}{\mu_1\mu_2}, 0\right)$.
Define
\begin{align}\label{R_0}
\Re_{0} = \frac{\beta_1S_0+\beta_2V_0}{\mu_3}
\end{align}
as the basic reproduction number.
The well known results for (\ref{ODEModel}) is
\begin{theorem}\cite{LiuTakeuchiShingoJTB2008}\label{Th13}
For system (\ref{ODEModel}), if $\Re_0<1$, $E_0$ is globally asymptotically stable; if $\Re_0>1$, system (\ref{ODEModel}) has a globally asymptotically stable positive equilibrium $E^* = (S^*,V^*,I^*)$ satisfy
\begin{equation}
\label{EE}\left\{
\begin{array}{l}
\vspace{2mm}
\displaystyle   \Lambda -  \beta_1S^*I^* - \mu_1 S^*=0,\\
\vspace{2mm}
\displaystyle   \alpha S^* -  \beta_2V^*I^*  - \mu_2 V^*=0,\\
\displaystyle   \beta_1S^*I^* + \beta_2V^*I^*  - \mu_3 I^*=0.
\end{array}\right.
\end{equation}
\end{theorem}

Now, we state our purpose of the current paper.
Letting $\varsigma=n+ct$ in system (\ref{Model}), where $c$ is wave speed, we arrive at
\begin{equation}
\label{WaveEqu}\left\{
\begin{array}{l}
\vspace{2mm}
\displaystyle   cS'(\varsigma)= \digamma[S](\varsigma) + \Lambda - \mu_1 S(\varsigma) - \beta_1S(\varsigma)I(\varsigma),\\
\vspace{2mm}
\displaystyle   c V'(\varsigma)= \digamma[V](\varsigma) + \alpha S(\varsigma) - \beta_2V(\varsigma)I(\varsigma) - \mu_2  V(\varsigma),\\
\displaystyle   c I'(\varsigma)= d\digamma[I](\varsigma) + \beta_1S(\varsigma)I(\varsigma) + \beta_2V(\varsigma)I(\varsigma) - \mu_3  I(\varsigma)
\end{array}\right.
\end{equation}
for all $\varsigma\in \mathbb{R}$, where $\digamma[\chi](\varsigma)\triangleq  \chi(\varsigma+1) - 2 \chi(\varsigma) +  \chi(\varsigma-1) $, $\mu_1 = \alpha + \mu$, $\mu_2 = \gamma_1 + \mu$ and $\mu_3 = \gamma + \mu$.
We want to find TWS satisfying:
\begin{equation}\label{Bound1}
\lim_{\varsigma\rightarrow-\infty}(S(\varsigma),  V(\varsigma),  I(\varsigma))=(S_0, V_0, 0),
\end{equation}
and
\begin{equation}\label{Bound2}
\lim_{\varsigma\rightarrow+\infty}(S(\varsigma),  V(\varsigma),  I(\varsigma))=(S^*, V^*, I^*).
\end{equation}

\subsection{Eigenvalue problem}
Linearizing the third equation of (\ref{WaveEqu}) at the $E_0$ yields
\begin{equation}
\label{LinearModel}
c I'(\varsigma)= d\digamma[ I](\varsigma) - \mu_3  I(\varsigma) + (\beta_1 S_0 + \beta_2 V_0) I(\varsigma).
\end{equation}
Let $ I(\varsigma)=e^{\mathfrak{r} \varsigma}$, we have
\begin{equation}\label{ODE}
d[e^\mathfrak{r} + e^{-\mathfrak{r}} - 2] - c\mathfrak{r} + (\beta_1 S_0 + \beta_2 V_0) - \mu_3 = 0.
\end{equation}
Denote
\begin{equation}\label{Delta}
\Delta(\mathfrak{r},c) = d[e^\mathfrak{r} + e^{-\mathfrak{r}} - 2] - c\mathfrak{r} + (\beta_1 S_0 + \beta_2 V_0) - \mu_3.
\end{equation}
By some calculations, for $\mathfrak{r}>0$ and $c>0$, we have
\begin{align*}
&\Delta(0,c) = (\beta_1 S_0 + \beta_2 V_0) - (\mu+\gamma),\ \ \ \lim_{c\rightarrow+\infty} \Delta (\mathfrak{r},c) = -\infty,\\
&\frac{\partial \Delta(\mathfrak{r}, c)}{\partial\mathfrak{r}} = d [e^\mathfrak{r} - e^{-\mathfrak{r}}] - c,\ \ \ \frac{\partial \Delta(\mathfrak{r}, c)}{\partial c}= -\mathfrak{r} < 0,\\
&\frac{\partial^2 \Delta(\mathfrak{r}, c)}{\partial\mathfrak{r}^2} = d [e^\mathfrak{r} + e^{-\mathfrak{r}}] > 0,\ \ \ \frac{\partial \Delta(\mathfrak{r}, c)}{\partial\mathfrak{r}}\bigg|_{(0,c)} = -c < 0.\\
\end{align*}
Therefore,
\begin{lemma}\label{WaveSpeed}
Let $\Re_{0}>1.$ There exist $\mathfrak{c}^*>0$ and $\mathfrak{r}^*>0$ such that
\[
\frac{\partial \Delta(\mathfrak{r}, c)}{\partial \mathfrak{r}}\bigg|_{(\mathfrak{r}^*,\mathfrak{c}^*)} = 0\ \ \textrm{and}\ \ \Delta(\mathfrak{r}^*,\mathfrak{c}^*) = 0.
\]
Furthermore,
\begin{description}
  \item[(i)]  $\Delta(\mathfrak{r}, c)>0$ for all $\mathfrak{r}$ if $0<c<\mathfrak{c}^*$;
  \item[(ii)] $\Delta(\mathfrak{r}, c)=0$ has only one positive real root $\mathfrak{r}^*$ if $c=\mathfrak{c}^*$;
  \item[(iii)] $\Delta(\mathfrak{r}, c)=0$ has two positive real roots $\mathfrak{r}_1,$ $\mathfrak{r}_2$ with $\mathfrak{r}_1<\mathfrak{r}^*<\mathfrak{r}_2$ if $c>\mathfrak{c}^*$.
\end{description}
\end{lemma}

\subsection{Sub- and super-solutions}
Fix $c>\mathfrak{c}^*$ and $\Re_0>1$, we show the following lemma.

\begin{lemma}\label{UpLow}
For $\varepsilon_i>0$ small enough and $M_i>0$ $(i=1,2,3)$ large enough, we define the following six functions:
\begin{equation*}%\label{UpLowSolution}
\left\{
\begin{array}{l}
\displaystyle   S^+(\varsigma)=S_0,\\
\displaystyle   V^+(\varsigma)=V_0,\\
\displaystyle   I^+(\varsigma) = e^{\mathfrak{r}_1 \varsigma},
\end{array}
\right.
\ \ \ \ \ \
\left\{
\begin{array}{l}
\displaystyle   S^-(\varsigma)=\max\{S_0(1-M_1 e^{\varepsilon_1 \varsigma}),0\},\\
\displaystyle   V^-(\varsigma)=\max\{V_0(1-M_2 e^{\varepsilon_2 \varsigma}),0\},\\
\displaystyle   I^-(\varsigma)=\max\{e^{\mathfrak{r}_1\varsigma}(1-M_3e^{\varepsilon_3 \varsigma}),0\}.
\end{array}\right.
\end{equation*}
%Then functions (\ref{UpLowSolution}) satisfy
Then they satisfy
\begin{equation}\label{up}
\left\{
\begin{array}{l}
%\displaystyle
c{S^+}'(\varsigma) \geq \digamma[S^+] + \Lambda - \mu_1 S^+ - \beta_1 S^+I^-,\qquad\ \\
%\displaystyle
c{V^+}'(\varsigma) \geq \digamma[V^+] + \alpha S^+ - \beta_2 V^+I^- - \mu_2 V^+,\qquad\quad\ \ \\
%\displaystyle
c {I^+}'(\varsigma) \geq d\digamma[I^+] + \beta_1 S^+I^+ + \beta_2 V^+I^+ - \mu_3  I^+,
\end{array}\right.
\end{equation}
and
\begin{subequations}\label{low}
\begin{numcases}{}
\label{S}
c{S^-}'(\varsigma) \leq \digamma[S^-] + \Lambda - \mu_1 S^- - \beta_1 S^-I^+,\ \ \ \ \ \ \ \ \  \ \varsigma\neq \varepsilon_1^{-1}\ln M_1^{-1}:=\mathfrak{X}_1,\\
\label{V}
c {V^-}'(\varsigma) \leq \digamma[V^-] + \alpha S^- - \beta_2  V^-I^+ - \mu_2  V^-,\ \ \ \ \ \ \varsigma\neq \varepsilon_2^{-1}\ln M_2^{-1}:=\mathfrak{X}_2,\\
\label{I}
c {I^-}'(\varsigma) \leq d\digamma[I^-] + \beta_1 S^-I^- + \beta_2 V^-I^- - \mu_3  I^-,\ \ \ \varsigma \neq \varepsilon_3^{-1}\ln M_3^{-1}:=\mathfrak{X}_3.
\end{numcases}
\end{subequations}
\end{lemma}
\begin{proof}
The proof of (\ref{up}) are trivial, so we omit the details.
Now, we focus on the proof of inequalities \eqref{low}.
If $\varsigma>\mathfrak{X}_1$, then equation (\ref{S}) holds since $S^-(\varsigma)=0$.
If $\varsigma<\mathfrak{X}_1$, then $S^-(\varsigma)=S_0(1-M_1 e^{\varepsilon_1 \varsigma})$,
and
\begin{align*}
&\ \digamma[S^-](\varsigma) + \Lambda - \mu S^-(\varsigma) - \beta_1 S^-(\varsigma)I^+(\varsigma) - c{S^-}'(\varsigma)\\
%= & e^{\varepsilon_1 \varsigma} S_0\left[-M_1 (e^{\varepsilon_1} + e^{-\varepsilon_1} - 2 - \mu - c \varepsilon_1) - \beta_1 f\left(e^{\mathfrak{r}_1\varsigma}\right) e^{-\varepsilon_1\varsigma} + \beta_1 M_1 \varepsilon_1 f\left(e^{\mathfrak{r}_1\varsigma}\right)\right]\\
\geq &\ e^{\varepsilon_1 \varsigma} S_0 \left[-M_1 (e^{\varepsilon_1} + e^{-\varepsilon_1} - 2 - \mu - c \varepsilon_1) -\beta_1 e^{\mathfrak{r}_1\varsigma}e^{-\varepsilon_1\varsigma}\right].
\end{align*}
Note taht $2 - e^{\varepsilon_1} - e^{-\varepsilon_1} - \mu - c \varepsilon_1 < 0$ and $e^{(\mathfrak{r}_1 - \varepsilon_1)\varsigma}\leq 1$ since $0<\varepsilon_1<\mathfrak{r}_1$is small enough and $\varsigma<\mathfrak{X}_1<0$. Thus, we need to choose a sufficiently large
\[
M_1 \geq \frac{\beta_1}{\mu + c \varepsilon_1 + e^{\varepsilon_1} + e^{-\varepsilon_1} - 2}.
\]
Then (\ref{S}) holds.

As for (\ref{I}), we choose $M_3$ such that $\frac{1}{\varepsilon_3}\ln M_3>\max\left\{\frac{1}{\varepsilon_1}\ln M_1,\frac{1}{\varepsilon_2}\ln M_2\right\}$.
If $\varsigma > \mathfrak{X}_3$, then (\ref{I}) holds since $ I^-(\varsigma)=0$. If $\varsigma < \mathfrak{X}_3$, then $ I^-(\varsigma)=e^{\mathfrak{r}_1\varsigma}(1-M_3e^{\varepsilon_3 \varsigma})$,
and (\ref{I}) is equivalent to
\begin{align*}
&\ d\digamma[I^-](\varsigma) + \beta_1 S^-(\varsigma)I^-(\varsigma) + \beta_2 V^-(\varsigma)I^-(\varsigma) - \mu_3  I^-(\varsigma) - c {I^-}'(\varsigma)\\
\geq &\ d \left[e^{\mathfrak{r}_1(\varsigma+1)}\left(1-M_3 e^{\varepsilon_3(\varsigma+1)}\right)+e^{\mathfrak{r}_1(\varsigma-1)}\left(1-M_3 e^{\varepsilon_3(\varsigma-1)}\right)-2e^{\mathfrak{r}_1\varsigma}\left(1-M_3 e^{\varepsilon_3\varsigma}\right)\right]\\
&\ +\beta_1 S_0e^{\mathfrak{r}_1 \varsigma}\left(1-M_1 e^{\varepsilon_1 \varsigma}\right)\left(1-M_3 e^{\varepsilon_3\varsigma}\right)+\beta_2 V_0e^{\mathfrak{r}_1 \varsigma}\left(1-M_2 e^{\varepsilon_2 \varsigma}\right)\left(1-M_3 e^{\varepsilon_3\varsigma}\right)\\
&\ -\mu_3 e^{\mathfrak{r}_1 \varsigma}\left(1-M_3 e^{\varepsilon_3\varsigma}\right) - c\mathfrak{r}_1 e^{\mathfrak{r}_1\varsigma} + c(\mathfrak{r}_1+\varepsilon_3) e^{(\mathfrak{r}_1+\varepsilon_3)\varsigma}\\
\geq &\ e^{\mathfrak{r}_1\varsigma}\Delta(\mathfrak{r}_1,c) - e^{(\mathfrak{r}_1+\varepsilon_3)\varsigma}M_3\Delta(\mathfrak{r}_1+\varepsilon_3,c) - \beta_1 S_0 M_1 e^{(\mathfrak{r}_1+\varepsilon_1)\varsigma} - \beta_2 V_0 M_2 e^{(\mathfrak{r}_1+\varepsilon_2)\varsigma}.
\end{align*}
Using the definition of $\Delta(\mathfrak{r},c)$ and noticing that $\Delta(\mathfrak{r}_1+\varepsilon_3,c)<0$, then it suffices to show that
%\[
%- e^{(\mathfrak{r}_1+\varepsilon_3)\varsigma}M_3\Delta(\mathfrak{r}_1+\varepsilon_3,c) \geq \beta_1 S_0 M_1 e^{(\mathfrak{r}_1+\varepsilon_1)\varsigma} + \beta_2 V_0 M_2 e^{(\mathfrak{r}_1+\varepsilon_2)\varsigma},
%\]
%that is
\[
- M_3\Delta(\mathfrak{r}_1+\varepsilon_3,c) \geq \beta_1 S_0 M_1 e^{(\varepsilon_1-\varepsilon_3)\varsigma} + \beta_2 V_0 M_2 e^{(\varepsilon_2-\varepsilon_3)\varsigma},
\]
which holds for $M_3$ is large enough.
%The above inequality holds for $M_3$ large enough, since the left-hand side vanishes and the right-hand side tends to infinity as $M_3\rightarrow+\infty$.
%This ends the proof.
\end{proof}

\section{Existence of traveling wave solutions}\label{Sec:Existence}% for $c>\mathfrak{c}^*$}

Let $\mathfrak{B}>-\mathfrak{X}_3>0$. Define
\begin{equation*}
\Gamma_\mathfrak{B} \triangleq \left\{(\phi, \varphi, \psi)\in C([-\mathfrak{B},\mathfrak{B}],\mathbb{R}^3)\left|
\begin{array}{l}
\vspace{2mm}
\displaystyle   S^-(\varsigma)\leq \phi(\varsigma) \leq S^+(\varsigma),\  V^-(\varsigma)\leq \varphi(\varsigma) \leq V^+(\varsigma),\\
\vspace{2mm}
\displaystyle   I^-(\varsigma)\leq \psi(\varsigma) \leq I^+(\varsigma)\ \ {\rm for}\ \ {\rm all}\ \ \varsigma\in[-\mathfrak{B},\mathfrak{B}],\\
\displaystyle   \phi(-\mathfrak{B})=S^-(-\mathfrak{B}),\ \ \varphi(-\mathfrak{B})=V^-(-\mathfrak{B}),\\
\displaystyle   \psi(-\mathfrak{B})=I^-(-\mathfrak{B}).
\end{array}\right.\right\}.
\end{equation*}
For any $(\phi,\varphi,\psi)\in C([-\mathfrak{B},\mathfrak{B}],\mathbb{R}^3)$,
define
\begin{equation*}
\label{hat1}
\hat{\phi}(\varsigma)=\left\{
\begin{array}{ll}
\displaystyle   \phi(\mathfrak{B}), &\mbox{for $\varsigma>\mathfrak{B}$,}
\\
\displaystyle   \phi(\varsigma), &\mbox{for $\varsigma\in[-\mathfrak{B},\mathfrak{B}]$,}
\\
\displaystyle   S^-(\varsigma), &\mbox{for $\varsigma< -\mathfrak{B}$,}\\
\end{array}\right.\ \ \
\hat{\varphi}(\varsigma)=\left\{
\begin{array}{ll}
\displaystyle   \varphi(\mathfrak{B}), &\mbox{for $\varsigma>\mathfrak{B}$,}
\\
\displaystyle   \varphi(\varsigma), &\mbox{for $\varsigma\in[-\mathfrak{B},\mathfrak{B}]$,}
\\
\displaystyle   V^-(\varsigma), &\mbox{for $\varsigma< -\mathfrak{B}$,}\\
\end{array}\right.
\end{equation*}
and
\begin{equation*}
\label{hat2}
\hat{\psi}(\varsigma)=\left\{
\begin{array}{ll}
\displaystyle   \psi(\mathfrak{B}), &\mbox{for $\varsigma>\mathfrak{B}$,}
\\
\displaystyle   \psi(\varsigma), &\mbox{for $\varsigma\in[-\mathfrak{B},\mathfrak{B}]$,}
\\
\displaystyle   I^-(\varsigma), &\mbox{for $\varsigma< -\mathfrak{B}$.}\\
\end{array}\right.
\end{equation*}
For $(\phi,\varphi,\psi)\in \Gamma_\mathfrak{B}$, consider %the following truncated initial problem:
\begin{equation}
\label{TruPro}\left\{
\begin{array}{l}
\vspace{2mm}
\displaystyle   cS'(\varsigma) + (2+\mu_1+\rho_1)S(\varsigma) = \hat{\phi}(\varsigma+1) + \hat{\phi}(\varsigma-1) + \Lambda + \rho_1 \phi(\varsigma) - \beta_1 \phi(\varsigma)\psi(\varsigma) := H_1(\psi,\varphi,\psi),\\
\vspace{2mm}
\displaystyle   cV'(\varsigma) + (2+\mu_2+\rho_2)V(\varsigma) = \hat{\varphi}(\varsigma+1) + \hat{\varphi}(\varsigma-1) + \alpha\phi(\varsigma) + \rho_2 \varphi - \beta_2 \varphi(\varsigma)\psi(\varsigma) := H_2(\psi,\varphi,\psi),\\
\vspace{2mm}
\displaystyle   cI'(\varsigma) + (2d+\mu_1)I(\varsigma) = \hat{\psi}(\varsigma+1) + \hat{\psi}(\varsigma-1) + \beta_1 \phi(\varsigma)\psi(\varsigma) + \beta_2 \varphi(\varsigma)\psi(\varsigma) := H_3(\psi,\varphi,\psi),\\
\displaystyle   (S,V,I)(-\mathfrak{B}) = (S^-,V^-,I^-)(-\mathfrak{B}),
\end{array}\right.
\end{equation}
where $\rho_1$ is large enough such that $\rho_1 \phi - \beta_1 \phi\psi$ is nondecreasing on $\phi$ and $\rho_2$ is large enough such that $\rho_2 \varphi - \beta_2 \varphi\psi$ is nondecreasing on $\varphi$. %By the standard ordinary differential equation theory,
Clearly, system (\ref{TruPro}) has a unique solution $(S_\mathfrak{B}(\varsigma),V_\mathfrak{B}(\varsigma),I_\mathfrak{B}(\varsigma))\in C^1([-\mathfrak{B},\mathfrak{B}],\mathbb{R}^3)$. Define
\[
\mathcal{A} = (\mathcal{A}_1,\mathcal{A}_2,\mathcal{A}_3):\Gamma_\mathfrak{B}\rightarrow C^1\left([-\mathfrak{B},\mathfrak{B}],\mathbb{R}^3\right)
\]
by
\[
S_\mathfrak{B}(\varsigma)=\mathcal{A}_1(\phi,\varphi,\psi)(\varsigma),\ \ V_\mathfrak{B}(\varsigma)=\mathcal{A}_2(\phi,\varphi,\psi)(\varsigma)\ \ {\rm and}\ \ I_\mathfrak{B}(\varsigma)=\mathcal{A}_3(\phi,\varphi,\psi)(\varsigma).
\]

%Next we show that the operator $\mathcal{A}$ has a fixed point in $\Gamma_\mathfrak{B}$.

\begin{lemma}\label{OA}
The operator $\mathcal{A}$ maps $\Gamma_\mathfrak{B}$ into itself and it is completely continuous.
\end{lemma}
\begin{proof}
Firstly, it is easy to show $\mathcal{A}$ maps $\Gamma_\mathfrak{B}$ into $\Gamma_\mathfrak{B}$ by Lemma \ref{UpLow}, so we omit the details.
Next, we focus on the second part of Lemma \ref{OA}.
For $i=1,2$, suppose that $(\phi_i(\varsigma),\varphi_i(\varsigma),\psi_i(\varsigma))\in\Gamma_\mathfrak{B}$ with
\[
S_{X,i}(\varsigma)=\mathcal{A}_1(\phi_i(\varsigma),\varphi_i(\varsigma),\psi_i(\varsigma)),\ \ V_{X,i}(\varsigma)=\mathcal{A}_2(\phi_i(\varsigma),\varphi_i(\varsigma),\psi_i(\varsigma)),
\]
and
\[
I_{X,i}(\varsigma)=\mathcal{A}_3(\phi_i(\varsigma),\varphi_i(\varsigma),\psi_i(\varsigma)).
\]
%We show that the operator $\mathcal{A}$ is continuous. By d
Direct calculation yields
\[
S_\mathfrak{B}(\varsigma) = S^-(-\mathfrak{B}) e ^{-\frac{2+\mu_1+\rho_1}{2}(\varsigma+\mathfrak{B})} + \frac{1}{c}\int_{-\mathfrak{B}}^\varsigma e ^{\frac{2+\mu_1+\rho_1}{2}(\tau+\mathfrak{B})}H_1(\phi,\varphi,\psi)(\tau)\d \tau,
\]
\[
V_\mathfrak{B}(\varsigma) = V^-(-\mathfrak{B}) e ^{-\frac{2+\mu_2+\rho_2}{2}(\varsigma+\mathfrak{B})} + \frac{1}{c}\int_{-\mathfrak{B}}^\varsigma e ^{\frac{2+\mu_2+\rho_2}{2}(\tau+\mathfrak{B})}H_2(\phi,\varphi,\psi)(\tau)\d \tau
\]
and
\[
I_\mathfrak{B}(\varsigma) = I^-(-\mathfrak{B}) e ^{-\frac{2d+\mu_3}{2}(\varsigma+\mathfrak{B})} + \frac{1}{c}\int_{-\mathfrak{B}}^\varsigma e ^{\frac{2d+\mu_3}{2}(\tau+\mathfrak{B})}H_3(\phi,\varphi,\psi)(\tau)\d \tau,
\]
For $i=1,2$ and any $(\phi_i, \varphi_i, \psi_i)\in\Gamma_\mathfrak{B}$, we have
\begin{align*}
|\phi_1(\varsigma)\psi_1(\varsigma) - \phi_2(\varsigma)\psi_2(\varsigma)|
\leq & \ |\phi_1(\varsigma)\psi_1(\varsigma) - \phi_1(\varsigma)\psi_2(\varsigma)|+|\phi_1(\varsigma)\psi_2(\varsigma) - \phi_2(\varsigma)\psi_2(\varsigma)|\\
\leq & \ S_0 \max_{\varsigma\in[-\mathfrak{B},\mathfrak{B}]}|\psi_1(\varsigma)-\psi_2(\varsigma)| + e^{\mathfrak{r}_1 \mathfrak{B}} \max_{\varsigma\in[-\mathfrak{B},\mathfrak{B}]}|\phi_1(\varsigma)-\phi_2(\varsigma)|.
\end{align*}
%and
%\begin{align*}
%|\varphi_1(\varsigma)\psi_1(\varsigma) - \varphi_2(\varsigma)\psi_2(\varsigma)|\ \leq & \ V_0 \max_{\varsigma\in[-\mathfrak{B},\mathfrak{B}]}|\psi_1(\varsigma)-\psi_2(\varsigma)| + e^{\mathfrak{r}_1 X} \max_{\varsigma\in[-\mathfrak{B},\mathfrak{B}]}|\varphi_1(\varsigma)-\varphi_2(\varsigma)|
%\end{align*}
Hence,
\begin{align*}
&\ c(S_{\mathfrak{B},1}'(\varsigma)-S_{\mathfrak{B},2}'(\varsigma))+(2+\mu_1)(S_{\mathfrak{B},1}(\varsigma)-S_{\mathfrak{B},2}(\varsigma))\\
\leq &\ |(\hat{\phi}_1(\varsigma+1)-\hat{\phi}_2(\varsigma+1))| + |(\hat{\phi}_1(\varsigma-1)-\hat{\phi}_2(\varsigma-1))| + \beta_1|\phi_1(\varsigma)\psi_1(\varsigma) - \phi_2(\varsigma)\psi_2(\varsigma)|\\
\leq &\ S_0 \max_{\varsigma\in[-\mathfrak{B},\mathfrak{B}]}|\psi_1(\varsigma)-\psi_2(\varsigma)| + \left(2+e^{\mathfrak{r}_1 \mathfrak{B}}\right)\max_{\varsigma\in[-\mathfrak{B},\mathfrak{B}]}|\phi_1(\varsigma)-\phi_2(\varsigma)|.
\end{align*}
Hence, the operator $\mathcal{A}$ is continuous by some similar arguments with $V_\mathfrak{B}$ and $I_\mathfrak{B}$. Moreover, $S_\mathfrak{B}'$, $V_\mathfrak{B}'$ and $I_\mathfrak{B}'$ are bounded by (\ref{TruPro}).
Thus, the operator $\mathcal{A}$ is completely continuous.
\end{proof}

By using Schauder's fixed point theorem, there exists $(S_\mathfrak{B},V_\mathfrak{B},I_\mathfrak{B})\in\Gamma_\mathfrak{B}$ such that
\[
(S_\mathfrak{B}(\varsigma),V_\mathfrak{B}(\varsigma),I_\mathfrak{B}(\varsigma)) = \mathcal{A}(S_\mathfrak{B},V_\mathfrak{B},I_\mathfrak{B})(\varsigma)
\]
for $\varsigma\in[-\mathfrak{B},\mathfrak{B}]$. Next, we give some prior estimate for $(S_\mathfrak{B},V_\mathfrak{B},I_\mathfrak{B})$.
Define
\[
C^{1,1}([-\mathfrak{B},\mathfrak{B}])=\{\upsilon\in C^1([-\mathfrak{B},\mathfrak{B}])\ |\ \upsilon,\upsilon' \textrm{are Lipschitz continuous}\}
\]
with
\begin{gather*}
  \|\upsilon\|_{C^{1,1}([-\mathfrak{B},\mathfrak{B}])}=\max_{x\in[-\mathfrak{B},\mathfrak{B}]}|\upsilon|+\max_{x\in[-\mathfrak{B},\mathfrak{B}]}|\upsilon'|+
  \sup_{\begin{subarray}{c}  x,y\in [-\mathfrak{B},\mathfrak{B}]   \\
                             x\neq y
        \end{subarray}}\frac{|\upsilon'(x)-\upsilon'(y)|}{|x-y|}.
\end{gather*}
\begin{lemma}\label{LemC}
There exists constant $\mathcal{C}(\mathcal{X})>0$ such that
\begin{equation*}
\|S_\mathfrak{B}\|_{C^{1,1}([-\mathcal{X},\mathcal{X}])}\leq \mathcal{C}(\mathcal{X}),\ \ \|V_\mathfrak{B}\|_{C^{1,1}([-\mathcal{X},\mathcal{X}])}\leq \mathcal{C}(\mathcal{X})\ \ {\rm and}\ \ \|I_\mathfrak{B}\|_{C^{1,1}([-\mathcal{X},\mathcal{X}])}\leq \mathcal{C}(\mathcal{X})
\end{equation*}
for $\mathcal{X}<\mathfrak{B}$.
\end{lemma}
\begin{proof}
Since $(S_\mathfrak{B},V_\mathfrak{B},I_\mathfrak{B})$ is the fixed point of $\mathcal{A},$ one has
\begin{equation}
\label{FixEqu}\left\{
\begin{array}{l}
\vspace{2mm}
\displaystyle   cS_\mathfrak{B}'(\varsigma) = \hat{S}_\mathfrak{B}(\varsigma+1) + \hat{S}_\mathfrak{B}(\varsigma-1) - (2+\mu_1)S_\mathfrak{B}(\varsigma) + \Lambda - \beta_1 S_\mathfrak{B}(\varsigma)I_\mathfrak{B}(\varsigma),\\
\vspace{2mm}
\displaystyle   cV_\mathfrak{B}'(\varsigma) = \hat{V}_\mathfrak{B}(\varsigma+1) + \hat{V}_\mathfrak{B}(\varsigma-1) - (2+\mu_2)V_\mathfrak{B}(\varsigma) + \alpha S_\mathfrak{B}(\varsigma) - \beta_2 V_\mathfrak{B}(\varsigma)I_\mathfrak{B}(\varsigma),\\
\displaystyle   cI_\mathfrak{B}'(\varsigma) = d\hat{I}_\mathfrak{B}(\varsigma+1) + d\hat{I}_\mathfrak{B}(\varsigma-1) - (2d+\mu_3)I_\mathfrak{B}(\varsigma) + (\beta_1 S_\mathfrak{B}(\varsigma) + \beta_2 V_\mathfrak{B}(\varsigma))I_\mathfrak{B}(\varsigma),
\end{array}\right.
\end{equation}
where
\begin{equation*}
\label{hatSV}
\hat{S}_\mathfrak{B}(\varsigma)=\left\{
\begin{array}{ll}
\displaystyle   S_\mathfrak{B}(\mathfrak{B}), &\mbox{for $\varsigma>\mathfrak{B}$,}
\\
\displaystyle   S_\mathfrak{B}(\varsigma), &\mbox{for $\varsigma\in[-\mathfrak{B},\mathfrak{B}]$,}
\\
\displaystyle   S^-(\varsigma), &\mbox{for $\varsigma< -\mathfrak{B}$,}\\
\end{array}\right.\ \ \
\hat{V}_\mathfrak{B}(\varsigma)=\left\{
\begin{array}{ll}
\displaystyle   V_\mathfrak{B}(\mathfrak{B}), &\mbox{for $\varsigma>\mathfrak{B}$,}
\\
\displaystyle   V_\mathfrak{B}(\varsigma), &\mbox{for $\varsigma\in[-\mathfrak{B},\mathfrak{B}]$,}
\\
\displaystyle   V^-(\varsigma), &\mbox{for $\varsigma< -\mathfrak{B}$,}\\
\end{array}\right.
\end{equation*}
and
\begin{equation*}
\label{hatI}
\hat{I}_\mathfrak{B}(\varsigma)=\left\{
\begin{array}{ll}
\displaystyle   I_\mathfrak{B}(\mathfrak{B}), &\mbox{for $\varsigma>\mathfrak{B}$,}
\\
\displaystyle   I_\mathfrak{B}(\varsigma), &\mbox{for $\varsigma\in[-\mathfrak{B},\mathfrak{B}]$,}
\\
\displaystyle   I^-(\varsigma), &\mbox{for $\varsigma< -\mathfrak{B}$.}\\
\end{array}\right.
\end{equation*}
Since $0\leq S_\mathfrak{B}(\varsigma)\leq S_0$, $0\leq V_\mathfrak{B}(\varsigma)\leq V_0$ and $0\leq I_\mathfrak{B}(\varsigma)\leq e^{\mathfrak{r}_1 \mathcal{X}}$ for all $\varsigma\in[-\mathcal{X},\mathcal{X}]$,
from (\ref{FixEqu}) we have
\[
|S_\mathfrak{B}'(\varsigma)|\leq \frac{4+\mu_1}{c}S_0 + \frac{\Lambda}{c} + \frac{\beta_1S_0}{c}e^{\mathfrak{r}_1 \mathcal{X}},
\]
\[
|V_\mathfrak{B}'(\varsigma)|\leq \frac{4+\mu_2}{c}V_0 + \frac{\alpha S_0}{c} + \frac{\beta_2V_0}{c}e^{\mathfrak{r}_1 \mathcal{X}},
\]
and
\[
|I_\mathfrak{B}'(\varsigma)|\leq \frac{4d+\mu_3 + (\beta_1S_0 + \beta_2V_0)}{c}e^{\mathfrak{r}_1 \mathcal{X}}.
\]
Hence,
\[
\|S_\mathfrak{B}\|_{C^{1}([-\mathcal{X},\mathcal{X}])}\leq C_1(\mathcal{X}),\ \ \|V_\mathfrak{B}\|_{C^{1}([-\mathcal{X},\mathcal{X}])}\leq C_1(\mathcal{X})\ \ {\rm and}\ \ \|I_\mathfrak{B}\|_{C^{1}([-\mathcal{X},\mathcal{X}])}\leq C_1(\mathcal{X}).
\]
for some constant $C_1(\mathcal{X}) > 0$.
It follows from \cite{ZhangWuIJB2019} that $|\hat{S}_\mathfrak{B}(\varsigma+1)-\hat{S}_\mathfrak{B}(\eta+1)|\leq C_1(\mathcal{X})|\varsigma-\eta|$ and $|\hat{S}_\mathfrak{B}(\varsigma-1)-\hat{S}_\mathfrak{B}(\eta-1)|\leq C_1(\mathcal{X})|\varsigma-\eta|$ for all $\varsigma,\eta\in[-\mathcal{X},\mathcal{X}]$. Furthermore
\begin{align*}
&\ |\beta_1 S_\mathfrak{B}(\varsigma)I_\mathfrak{B}(\varsigma) - \beta_1 S_\mathfrak{B}(\eta)I_\mathfrak{B}(\eta)|\\
\leq &\ |\beta_1 S_\mathfrak{B}(\varsigma)I_\mathfrak{B}(\varsigma) - \beta_1 S_\mathfrak{B}(\varsigma)I_\mathfrak{B}(\eta)|+|\beta_1 S_\mathfrak{B}(\varsigma)I_\mathfrak{B}(\eta) - \beta_1 S_\mathfrak{B}(\eta)I_\mathfrak{B}(\eta)|\\
\leq &\ \beta_1 C_1(\mathcal{X})\left(|S_\mathfrak{B}(\varsigma)-S_\mathfrak{B}(\eta)| + |I_\mathfrak{B}(\varsigma)-I_\mathfrak{B}(\eta)|\right)
\end{align*}
for all $\varsigma,\eta\in[-\mathcal{X},\mathcal{X}]$. Thus, $\|S_\mathfrak{B}\|_{C^{1,1}([-\mathcal{X},\mathcal{X}])}\leq \mathcal{C}(\mathcal{X})$ for some constant $\mathcal{C}(\mathcal{X}) > 0$.
Similarly,
\[
|V_\mathfrak{B}\|_{C^{1,1}([-\mathcal{X},\mathcal{X}])}\leq \mathcal{C}(\mathcal{X})\ \ {\rm and}\ \ \|I_\mathfrak{B}\|_{C^{1,1}([-\mathcal{X},\mathcal{X}])}\leq \mathcal{C}(\mathcal{X}).
\]
for any $\mathcal{X}<\mathfrak{B}$.
\end{proof}

With the help of Lemma \ref{LemC} and following from the standard arguments in \cite{ZhangWangLiuJNS2021}, we can conclude that $(S, V, I)$ is solution for system (\ref{WaveEqu}) with
\[
S^-\leq S(\varsigma)\leq S^+,\ \ V^-\leq V(\varsigma)\leq V^+,\ \ I^-\leq I(\varsigma)\leq I^+,\ \ \forall \varsigma\in\mathbb{R}.
\]

\section{Boundedness of traveling wave solution}\label{Sec:Bound}
In the following, we first show the boundedness of $(S, V, I)$.
\begin{lemma}\label{lem1}
The functions $S(\varsigma)$, $V(\varsigma)$ and $I(\varsigma)$ satisfy
\[
0<S(\varsigma)<S_0,\ \ 0<V(\varsigma)<V_0\ \ {\rm and}\ \  I(\varsigma)>0\ \ {\rm in}\ \ \mathbb{R}.
\]
\end{lemma}

\begin{proof}
Firstly, to show $S(\varsigma)>0$. If there exists some $\varsigma_0$ such that $S(\varsigma_0) = 0$, then $\digamma[S](\varsigma_0)\geq0$ and $S'(\varsigma_0) = 0$. Due to (\ref{WaveEqu}), we have
\[
0 = \digamma[S](\varsigma_0) + \Lambda > 0,
\]
which is a contradiction. Similarly, we have $ V(\varsigma)>0$ in $\mathbb{R}$.

Next, if there is $\varsigma_1$ such that $I(\varsigma_1) = 0$ and $I(\varsigma)>0$ for $\varsigma<\varsigma_1$.
From the third equation of (\ref{WaveEqu}), we have
\[
 I(\varsigma_1+1) +  I(\varsigma_1-1) = 0.
\]
Consequently, $ I(\varsigma_1+1) =  I(\varsigma_1-1) = 0$ since $I(\varsigma)\geq0$ in $\mathbb{R}$, which is a contradiction.% to the definition of $\varsigma_0$.

Lastly, we show that $S(\varsigma)<S_0$, if there exists $\varsigma_2$ such that $S(\varsigma_2) = S_0$, one has that
\[
0 = \digamma[S](\varsigma_2) - \beta_1 S(\varsigma_2) <0.
\]
This contradiction leads to $S(\varsigma)<S_0$. Similarly, we have $V(\varsigma)<V_0$ for all $\varsigma\in\mathbb{R}$.
This ends the proof.
\end{proof}

Now, we show the following four claims.

\textbf{Claim I}. The functions $\frac{ I(\varsigma\pm1)}{ I(\varsigma)}$ is bounded in $\mathbb{R}$.

To show this claim, we denote $\kappa := (2+\mu_3)/c$ and $U(\varsigma) := e^{\kappa \varsigma} I(\varsigma)$, one has that
\[
c U'(\varsigma) = e^{\kappa \varsigma}(c I'(\varsigma) + (\mu_3 + 2) I(\varsigma)>0,
\]
From the monotonicity of $U(\varsigma)$, we have
\[
\frac{ I(\varsigma-1)}{I(\varsigma)} < e^\kappa
\]
in $\mathbb{R}$.
Direct calculation yields
\begin{align}\label{Equ3}
\nonumber\left[e^{\kappa\varsigma} I(\varsigma)\right]' & = \frac{1}{c}e^{\kappa\varsigma}\left[d I(\varsigma+1) + d I(\varsigma-1) + (\beta_1S(\varsigma) + \beta_2 V(\varsigma)) I(\varsigma)\right]\\
& > \frac{d}{c}e^{\kappa\varsigma} I(\varsigma+1).
\end{align}
Integrating (\ref{Equ3}) over $[\varsigma,\varsigma+1]$ and using the monotonicity of $e^{\kappa \varsigma}$, one has
\begin{align*}
e^{\kappa(\varsigma+1)} I(\varsigma+1)\ > & \ e^{\kappa\varsigma} I(\varsigma) + \frac{d}{c}\int_\varsigma^{\varsigma+1}e^{\kappa s} I(s+1)\d s\\
> & \ e^{\kappa\varsigma} I(\varsigma) + \frac{d}{c}\int_\varsigma^{\varsigma+1}e^{\kappa (\varsigma+1)} I(\varsigma+1)e^{-\kappa}\d s\\
= & \ e^{-\kappa}\left[ I(\varsigma) + \frac{d}{c} I(\varsigma+1)\right].
\end{align*}
Hence,
\begin{equation}\label{Equ4}
\left[e^{\kappa\varsigma} I(\varsigma)\right]' > \left(\frac{d}{c}\right)^2 e^{-2\kappa}e^{\kappa(\varsigma+1)} I(\varsigma+1).
\end{equation}
Integrating (\ref{Equ4}) from $\varsigma-\frac{1}{2}$ to $\varsigma$ yields
\[
\frac{ I\left(\varsigma+\frac{1}{2}\right)}{ I(\varsigma)} < 2 \left(\frac{c}{d}\right)^2 e^{\frac{3}{2}\kappa},\ \ \forall\varsigma\in\mathbb{R}.
\]
Similarly, integrating (\ref{Equ4}) over $[\varsigma, \varsigma+\frac{1}{2}]$, we have
\[
\frac{ I(\varsigma+1)}{ I\left(\varsigma+\frac{1}{2}\right)} < 2 \left(\frac{c}{d}\right)^2 e^{\frac{3}{2}\kappa},\ \ \forall\varsigma\in\mathbb{R}.
\]
Thus
\[
\frac{ I(\varsigma+1)}{ I(\varsigma)} = \frac{ I\left(\varsigma+\frac{1}{2}\right)}{ I(\varsigma)} \frac{ I(\varsigma+1)}{ I\left(\varsigma+\frac{1}{2}\right)} < 4 \left(\frac{c}{d}\right)^4 e^{3\kappa},\ \ \forall\varsigma\in\mathbb{R}.
\]

\textbf{Claim II}. $\frac{ I'(\varsigma)}{ I(\varsigma)}$ is bounded in $\mathbb{R}$.

This claim is true because Claim I and the third equation of (\ref{WaveEqu}).

Choose a sequence $\{c_k,S_k, V_k, I_k\}$ of the TWS for (\ref{Model}) in a compact subinterval of $(0,\infty)$, we have the following claim.

\textbf{Claim III}. For a sequence $\{\varsigma_k\}$, we have $S(\varsigma_k)\rightarrow 0$ and $ V(\varsigma_k)\rightarrow 0$ as $k\rightarrow +\infty$ provided that $I(\varsigma_k)\rightarrow +\infty$ as $k\rightarrow +\infty$.

Let $\varsigma_k$ be a subsequence of $\{\varsigma_k\}_{k\in\mathbb{N}}$ with $I_k(\varsigma_k)\rightarrow+\infty$ and $S_k(\varsigma_k)\geq\varepsilon$ as $k\rightarrow+\infty$ in $\mathbb{R}$ for all $k\in\mathbb{N}$. Let $\tilde{c}>0$ be the lower bound of $\{c_k\}$ and we have
\[
S'_k(\varsigma)\leq \frac{2S_0+\Lambda}{\tilde{c}} := \delta_0\ \ \textrm{in}\ \ \mathbb{R}.
\]
We further denote $\delta = \frac{\varepsilon}{\delta_0}$, and we have
\[
S_k(\varsigma)\geq\frac{\varepsilon}{2},\ \ \ \forall\varsigma\in[\varsigma_k-\delta,\varsigma_k]\ \ \textrm{and}\ \ \forall k\in\mathbb{N}.
\]
Thanks to Claim II, there exists some $C_0>0$ such that
\[
\frac{ I_k(\varsigma_k)}{ I_k(\varsigma)} = \exp\left\{\int_{\varsigma}^{\varsigma_k}\frac{ I'_k(\sigma)}{ I_k(\sigma)}\d \sigma\right\}\leq e^{C_0\delta},\ \ \forall\varsigma\in[\varsigma_k-\delta, \varsigma_k]
\]
for all $k\in\mathbb{N}$. Thus
\[
\min_{\varsigma\in[\varsigma_k-\delta,\ \varsigma_k]} I_k(\varsigma)\geq e^{-C_0\delta} I_k(\varsigma_k),
\]
which give us
\[
\min_{\varsigma\in[\varsigma_k-\delta,\ \varsigma_k]} I_k(\varsigma) \rightarrow +\infty\ \ \textrm{as}\ \ k\rightarrow+\infty
\]
since $ I_k(\varsigma_k)\rightarrow+\infty$ as $k\rightarrow+\infty$.
Recalling the first equation of (\ref{WaveEqu}), and denote $\varsigma_k:=\min\limits_{\varsigma\in[\varsigma_k-\delta, \varsigma_k]} I_k(\varsigma)$, one can have
\[
\max_{\varsigma\in[\varsigma_k-\delta, \varsigma_k]}S'_k(\varsigma)\leq \delta_0 - \frac{\beta_1\varepsilon}{2}I_k(\varsigma_k)\rightarrow-\infty\ \ \textrm{as}\ \ k\rightarrow+\infty.
\]
Tnen,
\[
S'_k(\varsigma)\leq - \frac{2S_0}{\delta},\ \ \forall k\geq K\ \ \textrm{and}\ \ \varsigma\in[\varsigma_k-\delta, \varsigma_k].
\]
for some $K>0$.
%Since $S_k<S_0$ in $\mathbb{R}$ for each $k\in\mathbb{N}$,
Thus, we have $S_k(\varsigma_k)\leq-S_0$, $\forall k\geq K$, which is a contradiction. % with $S_k(\varsigma_k)\geq\varepsilon$ in $\mathbb{R}$ for all $k\in\mathbb{N}$ with some positive constant $\varepsilon$.
Similarly, we can show that $ V_k(\varsigma_k)\rightarrow0$ as $k\rightarrow+\infty$.

\textbf{Claim IV}. If $\limsup\limits_{\varsigma\rightarrow+\infty} I(\varsigma)=+\infty$, then $\lim\limits_{\varsigma\rightarrow+\infty} I(\varsigma)=+\infty$.

With a similar arguments in \cite[Lemma 3.4]{ChenGuoHamelNon2017}, we know that Claim IV is true.
We are now in position to show the boundedness of $I(\varsigma)$ by using Claim I-IV.

\begin{lemma}\label{lem5}
$I(\varsigma)$ is bounded in $\mathbb{R}$.
\end{lemma}

\begin{proof}
Suppose that $\limsup\limits_{\varsigma+\rightarrow\infty} I(\varsigma)=+\infty$, then$\lim\limits_{\varsigma\rightarrow+\infty}(S(\varsigma),V(\varsigma))=(0,0)$.
Denote $\theta(\varsigma)=\frac{I'(\varsigma)}{ I(\varsigma)}$, we have
\[
c\theta(\varsigma) = d e^{\int_{\varsigma}^{\varsigma+1}\theta(s)\d s} + d e^{\int_{\varsigma}^{\varsigma-1}\theta(s)\d s} - (2+\mu_3) + \beta_1S(\varsigma)+\beta_2 V(\varsigma),
\]
By using \cite[Lemma 3.4]{ChenGuoMA2003},
the finite limit of $\theta(\varsigma)$ at $+\infty$ exists and denoted by $\omega$, which is satisfying
\[
\Upsilon(\kappa,c) := d\left(e^\kappa + e^{-\kappa} - 2\right) -c\kappa - \mu_3 = 0.
\]
Clearly, $\Upsilon(\kappa,c) = 0$ has a unique positive real root $\kappa_0$.
From Lemma \ref{WaveSpeed}, we have
\[
d\left(e^{\mathfrak{r}_2} + e^{-\mathfrak{r}_2} - 2\right) -c\mathfrak{r}_2 - \mu_3 < 0,
\]
Recall the definition of $\mathfrak{r}_1$ and $\mathfrak{r}_2$, we have $\mathfrak{r}_2<\kappa_0$. Since $\lim\limits_{\varsigma\rightarrow+\infty}\theta(\varsigma) = \kappa_0$, then there exists $\tilde{\varsigma}$ such that
\[
 I(\varsigma)\geq C e{\left(\frac{\mathfrak{r}_2+\kappa_0}{2}\right)\varsigma}\ \ {\rm for}\ \ {\rm all}\ \ \varsigma\geq\tilde{\varsigma},
\]
with some constant $C$, which contradicts with $ I(\varsigma)\leq e^{\mathfrak{r}_1\varsigma}$ in $\mathbb{R}$ and $\mathfrak{r}_1<\kappa_0$.
The proof is done.
\end{proof}

The following lemma is to show that $I(\varsigma)$ cannot approach $0$.
\begin{lemma}\label{lem6}
There holds $\liminf\limits_{\varsigma\rightarrow+\infty} I(\varsigma)>0$.
\end{lemma}

\begin{proof}
We only need to show that if $ I(\varsigma)\leq\varepsilon_0$ for $\varepsilon_0>0$ is small enough, then $ I'(\varsigma)>0$ for all $\varsigma\in \mathbb{R}$.
If not, we can choose a sequence $\{\varsigma_k\}_{k\in\mathbb{N}}$ with $c_k\in(a,b)$ so that $ I(\varsigma_k)\rightarrow0$ as $k\rightarrow+\infty$ and $ I'(\varsigma_k)\leq0$, where $a$ and $b$ are two positive constants. Let
\[
S_k(\varsigma):= S(\varsigma_k+\varsigma),\ \  V_k(\varsigma):=  V(\varsigma_k+\varsigma)\ \ \textrm{and}\ \  I_k(\varsigma):=  I(\varsigma_k+\varsigma),
\]
then, $I_k(0)\rightarrow0$, $I_k(\varsigma)\rightarrow0$ and $I_k'(\varsigma)\rightarrow0$ locally uniformly in $\mathbb{R}$ as $k\rightarrow+\infty$,
and we can obtain that $S_\infty = S_0$ and $ V_\infty = V_0$ by a similar proof in \cite[Lemma 3.8]{ChenGuoHamelNon2017}.

Let $\pi_k(\varsigma):=\frac{ I_k(\varsigma)}{ I_k(0)}$, we have %By Lemma \ref{lem1}, and in the view of
\[
\pi_k'(\varsigma) = \frac{ I_k'(\varsigma)}{ I_k(0)} = \frac{ I_k'(\varsigma)}{ I_k(\varsigma)}\pi_k(\varsigma).
\]
Hence, $\pi_k(\varsigma)$ and $\pi_k'(\varsigma)$ also locally uniformly in $\mathbb{R}$ as $k\rightarrow+\infty$. Hence
\[
c_\infty\pi_\infty'(\varsigma)= d \digamma[\pi_\infty](\varsigma) + (\beta_1 S_0+\beta_2 V_0)\pi_\infty(\varsigma) - \mu_3 \pi_\infty(\varsigma)
\]
as $k\rightarrow+\infty$.
One can have $\pi_\infty(\varsigma)>0$ in $\mathbb{R}$. Indeed, if there is a $\varsigma_0$ such that $\pi_\infty(\varsigma_0)=0$, then ${\pi}'_\infty(\varsigma_0)=0$ and
then
\[
0 = d(\pi_\infty(\varsigma_0+1)+\pi_\infty(\varsigma_0-1)).
\]
Thus $\pi_\infty(\varsigma_0+1) = \pi_\infty(\varsigma_0-1) = 0$, it follows that $\pi_\infty(\varsigma_0+o) = 0$ for all $o\in\mathbb{Z}$. Recall that
$c_\infty\pi_\infty'(\varsigma) \geq  - \mu_3 \pi_\infty(\varsigma)$, then the map $\varsigma\mapsto \pi_\infty(\varsigma) e^{\frac{\mu_3\varsigma}{c_\infty}}$ is nondecreasing. Since it vanishes at $\varsigma_0+o$ for all $m\in\mathbb{Z}$, one can conclude that $\pi_\infty = 0$ in $\mathbb{R}$, which is a contradiction with $\pi_\infty(0) = 1$.

Denote $P(\varsigma):=\frac{\pi_\infty'(\varsigma)}{\pi_\infty(\varsigma)}$, one has that
\begin{equation}\label{Z}
c_\infty P(\varsigma)= d e^{\int^{\varsigma+1}_\varsigma P(s){\rm d} s}\d y + d e^{\int^{\varsigma-1}_\varsigma P(s){\rm d} s}\d y - 2 + \beta_1 S_0+\beta_2 V_0 - \mu_3.
\emph{}\end{equation}
Using \cite[Lemma 3.4]{ChenGuoMA2003}, $P(\varsigma)$ has finite limits $\omega_{\pm}$ and satisfy
% as $\varsigma\rightarrow\pm\infty$, where $\omega_\pm$ are roots of
\[
c_\infty \omega_\pm = d\left(e^{\omega_\pm} + e^{-\omega_\pm} -2\right) + \beta_1 S_0+\beta_2 V_0 - \mu_3.
\]
By Lemma \ref{WaveSpeed}, we know have $\omega_\pm>0$ and $\pi_\infty'(\pm\infty)$ are positive.
Moreover, one can have $\pi_\infty'(\varsigma)>0$ for all $\varsigma\in\mathbb{R}$. In fact, if there exists some $\varsigma^*$ such that $P(\varsigma^*) = \inf_{\mathbb{R}}P(\varsigma)$,
then $P(\varsigma^*) = 0$. Differentiating (\ref{Z}) yields
\[
c_\infty P'(\varsigma) = d(P(\varsigma+1) - P(\varsigma))\frac{\pi_\infty(\varsigma+1)}{\pi_\infty(\varsigma)} + d(P(\varsigma-1) - P(\varsigma))\frac{\pi_\infty(\varsigma-1)}{\pi_\infty(\varsigma)}.
\]
It follows that
\[
P(\varsigma^*) = P(\varsigma^*+1) = P(\varsigma^*-1).
\]
Hence $P(\varsigma^*) = P(\varsigma^*+\kappa)$ for all $\kappa\in\mathbb{Z}$. Then,
\[
\inf_{\mathbb{R}}P(\varsigma)\geq\min\{P(+\infty),P(-\infty)\}>0.
\]
Furthermore,
%So $\pi_\infty'(\varsigma)>0$. From the definition of $\pi_\infty(\varsigma)$, we have
\[
0<\pi_\infty'(0) = \lim_{k\rightarrow+\infty}\pi_k'(0) = \lim_{k\rightarrow+\infty}\frac{ I_k'(0)}{ I_k(0)}.
\]
Thus, $ I'(\varsigma_k) =  I_k'(0)>0$, which is a contradiction. %This completes the proof.
\end{proof}

\section{Convergence of the traveling wave solution}\label{Sec:Lyapunov}

In this section, we show the convergence of traveling wave solutions.
\begin{theorem}\label{theorem2}
If $\Re_0 > 1$, then for each $c > \mathfrak{c}^*$, system (\ref{Model}) has a TWS $(S(\varsigma), V(\varsigma), I(\varsigma))$ satisfying conditions (\ref{Bound1}) and (\ref{Bound2}).
\end{theorem}
\begin{proof}
In what following, we use $(S,V,I)$ short for $(S(\varsigma),V(\varsigma),I(\varsigma))$.
Define the following four functionals
\[
W_1(\varsigma) = c S^* g\left(\frac{S}{S^*}\right) + c V^* g\left(\frac{V}{V^*}\right) + c I^* g\left(\frac{I}{I^*}\right),
\]
\[
W_2(\varsigma) = \int_0^1 g\left(\frac{S(\varsigma-\sigma)}{S^*}\right)\d \sigma - \int_{-1}^0 g\left(\frac{S(\varsigma-\sigma)}{S^*}\right)\d \sigma
\]
\[
W_3(\varsigma) = \int_0^1 g\left(\frac{ V(\varsigma-\sigma)}{V^*}\right)\d \sigma - \int_{-1}^0 g\left(\frac{ V(\varsigma-\sigma)}{V^*}\right)\d \sigma
\]
and
\[
W_4(\varsigma) = \int_0^1 g\left(\frac{ I(\varsigma-\sigma)}{I^*}\right)\d \sigma - \int_{-1}^0 g\left(\frac{ I(\varsigma-\sigma)}{I^*}\right)\d \sigma,
\]
where $g(x)=x-1-\ln x$.
%Thanks to Lemma \ref{lem1}, Lemma \ref{lem5} and Lemma \ref{lem6}, the Lyapunov function $L(S, V, I)(\varsigma)$ is well defined and bounded from below.
%Next we show that $\varsigma\mapsto L(S, V, I)(\varsigma)$ is non-increasing.
The derivative of $W_1(\varsigma)$ is calculated as follows
\[
\frac{\d W_1(\varsigma)}{\d \varsigma} = \left(1-\frac{S^*}{S}\right) \digamma[S](\varsigma) + \left(1-\frac{V^*}{ V}\right) \digamma[ V](\varsigma) + \left(1-\frac{I^*}{ I}\right) d\digamma[ I](\varsigma) + \Sigma(\varsigma),
\]
where
\begin{align}\label{Theta}
\nonumber\Sigma(\varsigma) = & \left(1-\frac{S^*}{S}\right) \left(\Lambda - \mu_1 S - \beta_1 SI\right) + \left(1-\frac{V^*}{S}\right) \left(\alpha S - \beta_2  VI - \mu_2  V\right)\\
& + \left(1-\frac{I^*}{ I}\right) \left((\beta_1 S + \beta_2  V)I - \mu_3  I\right).
\end{align}
Since $(S^*,I^*,V^*)$ is the endemic equilibrium of system (\ref{Model}) and $\mu_1 = \mu + \alpha$,
one has
\begin{align*}
\Sigma(\varsigma) = &\ \mu S^* \left(2-\frac{S^*}{S}-\frac{S}{S^*}\right) + \mu_2 V^* \left(3-\frac{S^*}{S}-\frac{V}{V^*}-\frac{SV^*}{S^* V}\right)\\
& -\beta_1 S^* I^*\left[g\left(\frac{S^*}{S}\right) + g\left(\frac{S}{S^*}\right)\right] -\beta_2 V^* I^*\left[g\left(\frac{S^*}{S}\right) + g\left(\frac{SV^*}{S^* V}\right) + g\left(\frac{V}{V^*}\right)\right].
\end{align*}
Furthermore,
\begin{align*}
\frac{\d W_2(\varsigma)}{\d \varsigma} = &\frac{\d}{\d \varsigma} \left[\int_0^1 g\left(\frac{S(\varsigma-\sigma)}{S^*}\right)\d \sigma - \int_{-1}^0 g\left(\frac{S(\varsigma-\sigma)}{S^*}\right)\d \sigma\right]\\
= &  \int_0^1 \frac{\d}{\d \varsigma}g\left(\frac{S(\varsigma-\sigma)}{S^*}\right)\d \sigma - \int_{-1}^0 \frac{\d}{\d \varsigma}g\left(\frac{S(\varsigma-\sigma)}{S^*}\right)\d \sigma\\
= & - \int_0^1 \frac{\d}{\d \sigma}g\left(\frac{S(\varsigma-\sigma)}{S^*}\right)\d \sigma + \int_{-1}^0 \frac{\d}{\d \sigma}g\left(\frac{S(\varsigma-\sigma)}{S^*}\right)\d \sigma\\
= & 2 g\left(\frac{S}{S^*}\right) - g\left(\frac{S(\varsigma-1)}{S^*}\right) - g\left(\frac{S(\varsigma+1)}{S^*}\right).
\end{align*}
Similarly,
\[
\frac{\d W_3(\varsigma)}{\d \varsigma} = 2 g\left(\frac{ V}{V^*}\right) - g\left(\frac{ V(\varsigma-1)}{V^*}\right) - g\left(\frac{ V(\varsigma+1)}{V^*}\right)
\]
and
\[
\frac{\d W_4(\varsigma)}{\d \varsigma} = 2 g\left(\frac{ I}{I^*}\right) - g\left(\frac{ I(\varsigma-1)}{I^*}\right) - g\left(\frac{ I(\varsigma+1)}{I^*}\right).
\]
%By some calculations, it can be shown that
%\[
%\left(1-\frac{S^*}{S}\right) \digamma[S](\varsigma) + S^*\frac{\d W_2(S, V, I)(\varsigma)}{\d \varsigma} = - S^* \left[g\left(\frac{S(\varsigma-1)}{S}\right) + g\left(\frac{S(\varsigma+1)}{S}\right)\right],
%\]
%\[
%\left(1-\frac{V^*}{ V}\right) \digamma[ V](\varsigma) + V^*\frac{\d W_3(S, V, I)(\varsigma)}{\d \varsigma} = - V^* \left[g\left(\frac{ V(\varsigma-1)}{ V}\right) + g\left(\frac{ V(\varsigma+1)}{ V}\right)\right]
%\]
%and
%\[
%\left(1-\frac{I^*}{ I}\right) \digamma[ I](\varsigma) + I^*\frac{\d W_4(S, V, I)(\varsigma)}{\d \varsigma} = - I^* \left[g\left(\frac{ I(\varsigma-1)}{ I}\right) + g\left(\frac{ I(\varsigma+1)}{ I}\right)\right].
%\]
Now, we define a Lyapunov functional as
\[
\mathcal{V}(\varsigma) = W_1(\varsigma) + S^* W_2(\varsigma) + V^* W_3(\varsigma) + d I^* W_4(\varsigma),
\]
and
\begin{align*}
\frac{\d \mathcal{V}(\varsigma)}{\d \varsigma}
= &\mu S^* \left(2-\frac{S^*}{S}-\frac{S}{S^*}\right) + \mu_2 V^* \left(3-\frac{S^*}{S}-\frac{ V}{V^*}-\frac{SV^*}{S^* V}\right)\\
& -\beta_1 S^* I^*\left[g\left(\frac{S^*}{S}\right) + g\left(\frac{S}{S^*}\right)\right] -\beta_2 V^* I^*\left[g\left(\frac{S^*}{S}\right) + g\left(\frac{SV^*}{S^* V}\right) + g\left(\frac{V}{V^*}\right)\right]\\
&- S^* \left[g\left(\frac{S(\varsigma-1)}{S}\right) + g\left(\frac{S(\varsigma+1)}{S}\right)\right] - V^* \left[g\left(\frac{ V(\varsigma-1)}{ V}\right) + g\left(\frac{ V(\varsigma+1)}{ V}\right)\right]\\
& - dI^* \left[g\left(\frac{ I(\varsigma-1)}{ I}\right) + g\left(\frac{ I(\varsigma+1)}{ I}\right)\right].
\end{align*}
Recall that $g(x)\geq0$. Hence, the map $\varsigma\mapsto \mathcal{V}(\varsigma)$ is non-increasing. Choosing $\{\varsigma_k\}_{k\geq 0}$ as an increasing sequence with $\varsigma_k>0$ and $\varsigma_k\rightarrow+\infty$ as $k\rightarrow+\infty$. Let
$$\{S_k(\varsigma)=S(\varsigma+\varsigma_k)\}_{k\geq 0},\ \ \{ V_k(\varsigma)= V(\varsigma+\varsigma_k)\}_{k\geq 0}\ \ \textrm{and}\ \ \{ I_k(\varsigma)= I(\varsigma+\varsigma_k)\}_{k\geq 0}.$$
%Since the functions $S$, $ V$ and $ I$ are bounded, system (\ref{WaveEqu}) implies that the functions
Since $S$, $V$ and $I$ have bounded derivatives then the sequences of functions $\{S_k(\varsigma)\}$, $\{V_k(\varsigma)\}$ and $\{I_k(\varsigma)\}$ converge in $C_{loc}^{\infty}(\mathbb{R})$ as $k\rightarrow+\infty$ by Arzela-Ascoli theorem, up to extraction of a subsequence,
we assume that the sequences $\{S_k(\varsigma)\}$, $\{V_k(\varsigma)\}$ and $\{I_k(\varsigma)\}$ convergence to some nonnegative $C^\infty$ functions $S_\infty$, $ V_\infty$ and $ I_\infty.$ Since $\mathcal{V}(S, V, I)(\varsigma)$ is bounded from below,
then there exists $M_0$ and some large $k$ such that
\[
M_0\leq \mathcal{V}(S_k, V_k, I_k)(\varsigma)=\mathcal{V}(S, V, I)(\varsigma+\varsigma_k)\leq \mathcal{V}(S, V, I)(\varsigma).
\]
Hence, there exists some $\delta\in \mathbb{R}$ such that $\lim\limits_{k\rightarrow\infty} \mathcal{V}(S_k, V_k, I_k)(\varsigma)=\delta$, $\forall\varsigma\in \mathbb{R}$. Using Lebegue dominated convergence theorem, one has that
\[
\lim_{k\rightarrow+\infty}\mathcal{V}(S_k, V_k, I_k)(\varsigma)=\mathcal{V}(S_\infty, V_\infty, I_\infty)(\varsigma),\ \varsigma\in \mathbb{R}.
\]
Thus
\[
\mathcal{V}(S_\infty, V_\infty, I_\infty)(\varsigma)=\delta.
\]
Recall that $\frac{\d \mathcal{V}}{\d \varsigma}=0$ if and only if $S\equiv S^*$, $ V\equiv V^*$ and $ I\equiv I^*$, which finishes the proof.
%it follows that
%\[
%(S_\infty, V_\infty, I_\infty)\equiv (S^*,V^*,I^*).
%\]
%This completes the proof.
\end{proof}

%\begin{remark}
%We can solve the open problem proposed in \cite{ChenGuoHamelNon2017} by using the same idea in the proof of Theorem \ref{theorem2}, that is, we can construct an appropriate Lyapunov function to show the convergence of the TWS of model \eqref{PreModel}.
%\end{remark}

\begin{remark}
For the case $c=\mathfrak{c}^*$, we can obtain the existence of TWS by using a similar approximation technique used in \cite[Section 4]{ChenGuoHamelNon2017}. The TWS for $c=\mathfrak{c}^*$ is also satisfy \eqref{Bound1} and \eqref{Bound2} since the Lyapunov functional is independent of $c$.
\end{remark}

\section{Discussion}\label{Sec:Dis}
In this paper, we proposed a discrete diffusive vaccination epidemic model (i.e., system (\ref{Model})), which seems to be more realistic than the non-delayed model (1.2). Employing Schauder’s fixed point theorem and Lyapunov functional, we obtain the existence of nontrivial positive TWS, which is connecting two different equilibrium. Our research examines the conditions (i.e. basic reproduction number) under which an infectious disease can spread, even this disease has a vaccine.

Now we finish this section with some explanations from the perspective of epidemiology.
Assume that $(\hat{\mathfrak{r}},\hat{c})$ is a root of $\Delta(\mathfrak{r},c)=0$, by some calculations, we obtain
\[
\frac{\d \hat{c}}{\d \gamma_1} <0,\ \ \frac{\d \hat{c}}{\d d}  > 0,\
\frac{\d \hat{c}}{\d \beta_1} >0,\ \ \frac{\d \hat{c}}{\d \beta_2} >0\ \ {\rm and}\ \ \frac{\d \hat{c}}{\d \Re_0} >0.
\]
Mathematically, $\hat{c}$ is a decreasing on $\gamma_1$, while $\hat{c}$ is an increasing function on $d$, $\beta_1$ and $\beta_2$. From the biological point of view, this indicates the following three scenarios:
\begin{description}
  \item[I.] The more successful the vaccination, the slower the disease spreads;
  \item[II.] The faster the infected individuals move, the faster the disease spreads;
  \item[III.] The more effective the infections are, the faster the disease spreads.
\end{description}
Accordingly, a good understanding of the movement of the infected individuals and the vaccination rate of susceptible individuals could be important in disease control strategy.
In fact, as in the ordinary differential equation case in \cite{LiuTakeuchiShingoJTB2008}, the basic reproduction number $\Re_0$ is decreasing on $\gamma_1$, while $\Re_0$ is increasing on $\beta_1$ and $\beta_2$. Compared with \cite{LiuTakeuchiShingoJTB2008}, our study proposes a new explanation, which is to control the movement of the infected individuals.
Another important thing is the effectiveness of vaccination $\gamma_1$ is important than the vaccination rate $\alpha$, which explain the importance of complete vaccination.

In addition, as we know, numerical simulation can help us to see the influence of model parameters on the qualitative analysis of TWS. However, the topic of the present paper is focus on the mathematical theory analysis for the model, and we leave numerical simulation as our future work to fit some specific diseases.

\section*{Acknowledgements}
The authors would like to thank Professor Shigui Ruan from University of Miami for his suggestions.

\end{document}